\documentclass[11pt,reqno,a4paper]{amsart}
\usepackage[english]{babel}
\usepackage[applemac]{inputenc}
\usepackage[T1]{fontenc}
\usepackage{palatino}
\usepackage{amsfonts}
\usepackage{amsmath}
\usepackage{amssymb}
\usepackage{amsthm}
\usepackage{graphicx}
\usepackage{wrapfig}
\usepackage{type1cm}
\usepackage[bf]{caption}
\usepackage{esint}
\usepackage{version}
\usepackage[colorlinks = true, citecolor = black, linkcolor = black, urlcolor = black, pdfstartview=FitH]{hyperref}
\usepackage[usenames]{color}
\usepackage{caption}
\usepackage{tikz}
\usepackage{bbding, wasysym}
\usepackage{verbatim}
\usepackage{psfrag}
\usepackage{textpos}
  
\newcommand{\R}{\mathbb{R}}
\newcommand{\C}{\mathbb{C}}

\newcommand{\N}{\mathbb{N}}

\renewcommand{\S}{\mathbb{S}}

\newcommand{\cL}{\mathcal{L}}

\newcommand{\cG}{\mathcal{G}}
\newcommand{\cR}{\mathcal{R}}

\newcommand{\cP}{\mathcal{P}}

\newcommand{\Hd}{\dim_\mathrm{H}}

\newcommand{\Fd}{\dim_\mathrm{F}}

\newcommand{\Z}{\mathbb{Z}}

\newcommand{\spt}{\operatorname{spt}}

\renewcommand{\epsilon}{\varepsilon}
\renewcommand{\rho}{\varrho}
\renewcommand{\phi}{\varphi}

\newcommand{\Rea}{\operatorname{Re}}

\theoremstyle{plain}
\newtheorem{theorem}{Theorem}

\newtheorem{lemma}{Lemma}
\newtheorem{proposition}{Proposition}
\newtheorem{cor}{Corollary}
\newtheorem{corollary}[cor]{Corollary}

\newtheorem*{claim*}{Claim}

\theoremstyle{definition}
\newtheorem{definition}{Definition}

\newtheorem*{examples*}{Examples}
\newtheorem*{example*}{Example}

\newtheorem{question}{Question}

\newtheorem*{notations*}{Notations}
\newtheorem*{notation*}{Notation}

\numberwithin{equation}{section}
\numberwithin{figure}{section}
\numberwithin{theorem}{section}
\numberwithin{lemma}{section}
\numberwithin{proposition}{section}
\numberwithin{cor}{section}
\numberwithin{claim}{section}
\numberwithin{definition}{section}
\numberwithin{conjecture}{section}
\numberwithin{example}{section}
\numberwithin{remark}{section}
\numberwithin{notations}{section}
\numberwithin{notation}{section}
\numberwithin{question}{section}

\title{On Fourier analytic properties of graphs}
\author{Jonathan M. Fraser, Tuomas Orponen, and Tuomas Sahlsten}

\thanks{JMF is financially supported by an EPSRC doctoral training grant.  TO is financially supported by the Finnish National Doctoral Programme in Mathematics and its Applications. TS acknowledges the support from the Finnish Centre of Excellence in Analysis and Dynamics Research and Emil Aaltonen Foundation.}

\address{Mathematical Institute, University of St Andrews, North Haugh, St Andrews, Fife, KY16 9SS, Scotland} \email{jmf32@st-andrews.ac.uk}

\address{Department of Mathematics and Statistics, P.O. Box 68, 00014 University of Helsinki, Finland}
\email{tuomas.orponen@helsinki.fi} 

\address{Department of Mathematics, University of Bristol, University Walk, Clifton, Bristol, BS8 1TW, England}
\email{tuomas.sahlsten@bristol.ac.uk}


\begin{document}

\maketitle
\begin{abstract} We study the Fourier dimensions of graphs of real-valued functions defined on the unit interval $[0,1]$. Our results imply that the graph of fractional Brownian motion is almost surely not a Salem set, answering in part a question of Kahane from 1993, and that the graph of a Baire typical function in $C[0,1]$ has Fourier dimension zero.
\end{abstract}

\section{Introduction}

This paper is concerned with the decay of Fourier transforms of measures supported on graphs of real-valued functions defined on the unit interval $[0,1]$. Given such a function $f \colon [0,1] \to \R$, the \emph{graph of $f$} is, as usual, the set
$$ G_{f} := \{(x,f(x)) : x \in [0,1]\} \subset \R^{2}. $$
Suppose that $\mu$ is a Borel probability measure supported on $G_{f}$. How fast can the Fourier transform $\widehat{\mu}$ decay at infinity? To quantify this question, we look for exponents $s \geq 0$ such that
\begin{equation}\label{Fdim} |\widehat{\mu}(\xi)| \leq C|\xi|^{-s/2}, \qquad \xi \in \R^{2}, \end{equation}
for some constant $C > 0$. If $G_{f}$ supports a Borel probability measure $\mu$ satisfying \eqref{Fdim} for some exponent $s \leq 2$, we say that the \emph{Fourier dimension of $G_{f}$}, denoted by $\Fd G_{f}$, is at least $s$. Of course, the notion of Fourier dimension can be defined for all sets $K \subset \R^{2}$, not just graphs. In general, the number $\Fd K$ is the supremum over all exponents $s \leq 2$ such that \eqref{Fdim} holds for some Borel probability measure supported on $K$. This number never exceeds the Hausdorff dimension of $K$, denoted by $\Hd K$, see \cite[Section 12.17]{mattila} or \cite[Corollary 8.7]{wolff}.

To get a picture of the possible values the Fourier dimension can attain for graphs, one should keep in mind two `extremal' examples. The first one is the graph of a constant function; a horizontal line segment in $\R^{2}$. Then, no matter how one chooses a probability measure $\mu$ on $G_{f}$, the Fourier transform $\widehat{\mu}$ restricted to the $y$-axis will be a constant with absolute value one. This means that $\Fd G_{f} = 0$.

The opposite behaviour is manifest for graphs of smooth functions with non-vanishing second derivative. In this case it is well-known, see for instance \cite{kaufman}, that there exist non-zero measures supported on $G_{f}$, whose Fourier transforms satisfy \eqref{Fdim} with the exponent $s = 1$. Thus, $\Fd G_{f} = 1$.

So far we have seen that $\Fd G_{f}$ can attain the values zero and one. A straightforward application of the results in \cite{kaufman} shows that anything in between is possible as well:
\begin{proposition} \label{attainable} For any $s \in [0,1]$, there exists a function $f \in C[0,1]$ such that $\Fd G_{f} = s$. \end{proposition}

To the best of our knowledge, the existence of graphs with $\Fd G_{f} > 1$ is an open question. In his 1993 survey \cite{kahane}, Kahane writes: ``\emph{...proving almost sure roundedness for specific random sets is never easy and it remains an open program for most natural random sets: level sets and graphs of random functions in particular.}" The word `roundedness' has the following meaning here: in the terminology of \cite{kahane}, a set $K \subset \R^{2}$ is \emph{round}, if $$\Fd K = \Hd K.$$
In recent literature, such sets $K$ are often referred to as \emph{Salem sets}.  The question of whether graphs of random functions are Salem was raised again and formalised by Shieh and Xiao \cite[Question 2.15]{xiao} where they ask ``\emph{Are the graph and level sets of a stochastic process such as fractional Brownian motion Salem sets?}" and attribute the origin of the question to Kahane.

Perhaps the most classical example of a random process producing functions in $C[0,1]$ is one-dimensional Brownian motion. In 1953, Taylor \cite{taylor} proved that the graph of one-dimensional Brownian motion has Hausdorff dimension $3/2$ almost surely. Moreover, in 1977, Adler \cite{Adl} showed that the graph of fractional Brownian motion on $[0,1]$ with Hurst index $0 < H < 1$ has Hausdorff dimension $2-H$ almost surely. So, in order for these graphs to be Salem sets, or `round', also the Fourier dimension should be $2-H$ almost surely. Our first main result shows that this is definitely not the case:
\begin{theorem}\label{atMost} 
For any function $f : [0,1] \to \R$, we have
$$\Fd G_f \leq 1.$$
\end{theorem}

Unfortunately, we were unable to settle exactly `how round' the graph of fractional Brownian motion is; we conjecture that the Fourier dimension is almost surely one.  The fact that graphs of fractional Brownian motion are not Salem is in sharp contrast with the result of Kahane \cite{Ka1} that the \textit{image} of any Borel set $E \subset \R$ under fractional Brownian motion is a Salem set in $\R$ almost surely. Moreover, it was shown by Kahane in \cite{kahaneCourse} that the \textit{level sets} of ordinary Brownian motion are Salem almost surely and recently in \cite{FM} Fouch\'{e} and Mukeru extend the results for fractional Brownian motion.

Our second main result concerns the Fourier dimension of the graph of a \emph{typical} function in $C[0,1]$. Let us recall the notion of `typicality' before proceeding. In a general complete metric space $X$, a set $M \subset X$ is said to be \emph{meagre}, if it can be written as a countable union of nowhere dense sets, and a set $\cR \subseteq X$ is \emph{residual}, if $X \setminus \cR$ is meagre. A property is called \emph{typical} in the space $X$, if the set of points which have the property is residual.  Observe that $C[0,1]$ is a complete metric space when equipped with the $\sup$-norm $\| \cdot \|_\infty$, so talking of `typical properties in $C[0,1]$' makes sense.  

There are several previous results in the literature describing the properties of graphs of typical $C[0,1]$-functions.  Mauldin and Williams \cite{graphsums} observed that the graph of a typical function in $C[0,1]$ has Hausdorff dimension one, while Humke and Petruska \cite{humkepacking} showed that the graph of a typical function in $C[0,1]$ has packing dimension two.  More recently, Hyde \emph{et al.}  \cite{bairefunctions} proved that in the more general space $C(E)$, for some compact set $E \subseteq \mathbb{R}$, the typical lower box dimension of a graph is the lower box dimension of $E$, and the typical upper box dimension of a graph is the maximum value possible (the \emph{upper graph box dimension}); interestingly, this value may be strictly less than the upper box dimension of $E$ plus one.  The problem of computing the upper graph box dimension for an arbitrary set $E$ can be awkward and has been investigated further in \cite{graphdim}.  

Our second main result implies that the graph of a typical function in $C[0,1]$ has Fourier dimension zero. Again, this value is strictly smaller than the corresponding value for the Hausdorff dimension, which is one according to the result in \cite{graphsums}. 

\begin{theorem}\label{noDecay} 
For a typical function $f \in C[0,1]$, we have
$$ \limsup_{|\xi| \to \infty} |\widehat{\mu}(\xi)| \geq \frac{1}{5}$$
for any Borel probability measure $\mu$ supported on $G_{f}$. In particular, the Fourier dimension of the typical graph equals zero.
\end{theorem}

Following the idea of Mauldin and Williams in \cite[Theorem 2]{graphsums}, Theorem \ref{noDecay} has the following consequence:

\begin{corollary}
Any function $f \in C[0,1]$ can be written as a sum 
$$f = f_{1} + f_{2},$$ 
where $f_{1},f_{2} \in C[0,1]$ are functions such that $\widehat{\mu} \not\to 0$ for all Borel probability measures $\mu$ supported on either $G_{f_{1}}$ or $G_{f_{2}}$.
\end{corollary}

We remark that Theorems \ref{atMost} and \ref{noDecay} can be easily extended to the space $C(E)$, for any closed set $E \subseteq [0,1]$, whereas the $C(E)$ analogue for Proposition \ref{attainable} remains open. These points are elaborated on in Section \ref{spaceE}.

\section{Notation}

For a compact set $E \subset \R$, we write $C(E)$ for the space of all continuous functions $f : E \to \R$, endowed with the $\sup$-norm $\|\cdot\|_\infty$. If $K \subset \R^{2}$ is any set, we write $\cP(K)$ for the family of all Borel probability measures supported on $K$. The Hausdorff dimension of a set $K$ is denoted by $\Hd K$, see \cite[Definition 4.8]{mattila}. For $\mu \in \cP(\R^{2})$, we define the \textit{Fourier transform} $\widehat{\mu} \colon \R^{2} \to \C$ by
$$\widehat{\mu}(\xi) := \int_{\R^{2}} e^{-2\pi i x \cdot \xi} \, d\mu(x). $$
Similar notations and definitions are used for measures and their Fourier transforms in $\R$. Throughout the paper we write $a \lesssim b$, if $a \leq Cb$ for some constant $C \geq 1$. Should we wish to emphasize that $C$ depends on some parameter $p$, we may write $a \lesssim_{p} b$. With this notation, the \textit{Fourier dimension} of a set $K \subset \R^{2}$ is defined by
$$\Fd K = \sup\{0 \leq s \leq 2 : \text{there exists } \mu \in \cP(K) \text{ such that } |\widehat{\mu}(\xi)| \lesssim |\xi|^{-s/2} \text{ for } \xi \in \R^2\}.$$
This definition (with probability measures replaced by non-zero positive measures) appears for instance in \cite[Section 12.17]{mattila}.

\section{Proofs}

\subsection{Proof of Proposition \ref{attainable}} \label{Attainable}

We recite a theorem by Kaufman from \cite{kaufman}: \textit{if $\Gamma$ is a $C^{2}$ curve in the plane with positive curvature, then for each $s \in (0,1)$ there exists a compact set $S \subset \Gamma$ with $\Hd S = s$ and a positive measure $\mu$ on $S$ such that $|\widehat{\mu}(\xi)| \lesssim |\xi|^{-s/2}$ for $\xi \in \R^{2}$.}

In our situation, we let $\Gamma = G_{f}$ be the graph of any $C^{2}$ function $f \colon [0,1] \to \R$ with non-vanishing second derivative. Fixing $s \in (0,1)$, we choose the compact set $S \subset G_{f}$ as in Kaufman's theorem, and denote by $\pi(S) \subset [0,1]$ the projection of $S$ under the mapping $\pi(x,y) = x$. We first deform $f$ into a function $g \colon [0,1] \to \R$ as follows. For $t \in \pi(S)$, we set $g(t) = f(t)$. The complement of $\pi(S)$ in $[0,1]$ is a countable union of open intervals (possibly half-open in two cases). We require that 
\begin{itemize}
\item[(a)] $g$ is affine on these complementary intervals, and
\item[(b)] $g \in C[0,1]$. 
\end{itemize}
If $\{0,1\} \subseteq \pi(S)$, as we may assume, these conditions determine $g$ uniquely.

It remains to establish that $\Fd G_{g} = s$. Since the measure $\mu$ given by Kaufman's theorem is supported on $S \subset G_{g}$, we already have $\Fd G_{g} \geq s$. To prove the converse inequality, we have to show that no matter how we choose a measure $\nu \in \cP(G_{g})$, the Fourier transform $\widehat{\nu}$ cannot satisfy the uniform bound 
\begin{equation}\label{form7} |\widehat{\nu}(\xi)| \lesssim |\xi|^{-t/2}, \qquad \xi \in \R^{2}, \end{equation}
for any $t > s$. So, fix a measure $\nu \in \cP(G_{g})$. There are two cases: either $\spt \nu \subset S$, or $\spt \nu \not\subset S$. In the former case, the non-existence of numbers $t > s$ satisfying \eqref{form7} follows immediately from $\Hd S = s$, since the inequality \eqref{form7} always implies $\Hd S \geq t$.

In the latter case, we claim that $\widehat{\nu}$ cannot tend to zero at infinity. First, find a line segment $J \subset G_{g}$ with $\nu(J) > 0$: such an object exists, because $\spt \nu \not\subset S$. Let $e \in \S^1$ be a unit vector in the orthogonal complement of $J$, and consider the projection $\pi_{e} \colon \R^{2} \to \R$, defined by $\pi_{e}(x) = e \cdot x$. Writing $\nu_{e} := \pi_{e\sharp}\nu$ for the image of $\nu$ under $\pi_{e}$, it follows by inspecting the definitions of $\widehat{\nu}$ and $\widehat{\nu_{e}}$ that
\begin{displaymath} \widehat{\nu_{e}}(t) = \widehat{\nu}(te), \qquad t \in \R. \end{displaymath}
Moreover, the measure $\nu_{e}$ has an atom of mass $\nu(J) > 0$ concentrated at the singleton $\pi_{e}(J) \subset \R$. Combining these observations, the fact that $\widehat{\nu}$ does not tend to zero is a consequence of the following well-known general principle: if a probability measure on the real line has an atom, then its Fourier transform does not tend to zero at infinity. We sketch the argument for completeness.

It is clear that
\begin{equation}\label{form8} \sum_{t \in \R} \nu_{e}(\{t\})^{2} = \lim_{\lambda \searrow 0} \int \nu_{e}([t - \lambda, t + \lambda]) \, d\nu_{e}(t). \end{equation}
Now, let $\psi \colon \R \to \R$ be a positive rapidly decaying smooth function satisfying $\psi|_{[-1,1]} \geq 1$ and $\spt \widehat{\psi} \subset [-R,R]$ for some $R > 0$. Write $\psi_{\lambda}(t) := \psi(t/\lambda)$.  Using Parseval's formula and the convolution rule, we may continue \eqref{form8} as follows:
\begin{align*} \sum_{t \in \R} \nu_{e}(\{t\})^{2} & \leq \liminf_{\lambda \searrow 0} \int \nu_{e} \ast \psi_{\lambda}(t) \, d\nu_{e}(t)\\
& = \liminf_{\lambda \searrow 0} \int |\widehat{\nu_{e}}(\xi)|^{2}\widehat{\psi_{\lambda}}(\xi) \, d\xi\\
& \lesssim \liminf_{\lambda \searrow 0} \lambda \int_{-R/\lambda}^{R/\lambda} |\widehat{\nu_{e}}(\xi)|^{2} \, d\xi.  \end{align*}
Since $\widehat{\nu_{e}}(\{t\}) > 0$ for some $t \in \R$, the preceding chain of inequalities shows that $\widehat{\nu_{e}}$ cannot tend to zero at infinity. This completes the proof of Proposition \ref{attainable}.

\subsection{Proof of Theorem \ref{atMost}}

The proof is based on the following auxiliary result.

\begin{lemma}
\label{slicing}
 Let $\mu \in \cP(\R^{2})$ be a measure, whose Fourier transform satisfies the condition 
 $$|\widehat{\mu}(\xi)| \lesssim |\xi|^{-\tau/2}$$ 
 for some $\tau > 1$. Denote by $\mu_1$ the projection of $\mu$ onto the $x$-coordinate. Then $\mu_1 \ll \cL^{1}$. Moreover, there are $\cL^{1}$ positively many vertical lines $L_{t} := \{(t,y) : y \in \R\}$, $t \in \R$, such that 
$$ \Hd (L_{t} \cap \spt \mu) > 0. $$
\end{lemma}

Let us quickly see how to derive Theorem \ref{atMost} from the lemma. If there existed a graph $G_f$ with $\Fd G_{f} > 1$, then, by definition of Fourier dimension, we could find a measure $\mu \in \cP(G_f)$ satisfying 
$$|\widehat{\mu}(\xi)| \lesssim |\xi|^{-\tau/2}, \qquad \xi \in \R^2,$$ 
for some $\tau > 1$. Then Lemma \ref{slicing} tells us that positively many vertical lines intersect the support of $\mu$, in particular the graph $G_{f}$, in a set of positive dimension. This is absurd, since vertical lines cross graphs at no more than one point.

\begin{proof}[Proof of Lemma \ref{slicing}] The Fourier transform of $\mu_{1}$ has the formula
$$ \widehat{\mu}_1(t) = \widehat{\mu}(t,0), \quad t \in \R.$$
Combining this, the assumption on the decay of $|\widehat{\mu}|$ and Plancherel's formula implies that $\mu_1 \in L^{2}(\R)$ since $\tau > 1$. In particular, $\mu_1 \ll \cL^{1}$. According to \cite[Chapter 10]{mattila}, this means that the measure $\mu$ can be `sliced' using the vertical lines $L_{t}$, $t \in \R$. More precisely, there exists a collection of Borel measures $\nu_t$, $t \in \R$, such that
\begin{itemize}
\item[(a)] the measures $\nu_{t}$ are non-zero for $\cL^{1}$ positively many $t \in \R$, and  
\item[(b)] the support of $\nu_t$ is contained in the intersection $L_{t} \cap \spt \mu$,
\end{itemize}
and for all continuous functions $\varphi \colon \R^{2} \to \R$ we have the formula:
$$ \int \varphi \, d\nu_{t} = \lim_{\delta \searrow 0} \frac{1}{2\delta}\int\limits_{T(t,\delta)} \varphi \, d\mu.$$
Here $T(t,\delta)$ is the tube $\{(x,y) \in \R^2 : t - \delta \leq x \leq t + \delta\}$.

The energy integrals of the sliced measures $\nu_t$ were a central object of study in \cite{orponen}, and the following inequality is \cite[Lemma 3.10]{orponen}:

$$ \int_{\R} I_{s - 1}(\nu_t) \, dt \lesssim_{s} \int_{\R^{2}} |\pi_{2}(\xi)|^{s - 2}|\widehat{\mu}(\xi)|^{2} \, d\xi, \quad 1 < s < 2, $$
where $\pi_{2}(\xi) = \xi_{2}$ is the projection onto the second coordinate and
$$
I_{s-1}(\nu_t) = \int \int \frac{d\nu_t(x) \, d\nu_t(y)}{\lvert x-y \rvert^{s-1}}
$$
is the $(s-1)$-energy of $\nu_t$. Applying this and the decay bound for $|\widehat{\mu}|$ (in the form $|\widehat{\mu}(\xi)| \lesssim (1 + |\xi|)^{-\tau/2}$) for any $s \in (1,\tau)$ yields
\begin{equation}\label{form6} \int_{\R} I_{s - 1}(\nu_t) \, dt \lesssim_{s} \int_{\R^{2}} |\pi_{2}(\xi)|^{s - 2}(1 + |\xi|)^{-\tau} \, d\xi < \infty, \end{equation}
which means that $\cL^{1}$ almost all of the measures $\nu_t$, $t \in \R$, must have finite $(s - 1)$-energy. In particular, 
$$\Hd (L_{t} \cap \spt \mu) \geq \Hd (\spt \nu_{t}) \geq s - 1 > 0$$ 
for $\cL^{1}$ positively many $t \in \R$, using both (a) and (b). 

The finiteness of the latter integral in \eqref{form6} can be justified as follows. For any $e \in \S^1$, it is clear that
$$ \int_{\R^{2}} |\pi_{2}(\xi)|^{s - 2}(1 + |\xi|)^{-\tau} \, d\xi = \int_{\R^{2}} |\xi \cdot e|^{s - 2}(1 + |\xi|)^{-\tau} \, d\xi, $$
since $\pi_{2}(\xi) = \xi \cdot (0,1)$. If $\sigma$ is the uniformly distributed probability measure on $\S^1$, it follows that
$$ \int_{\R^{2}} |\pi_{2}(\xi)|^{s - 2}(1 + |\xi|)^{-\tau} \, d\xi = \int_{\R^{2}} \left(\int_{\S^1} \left|\frac{\xi}{|\xi|} \cdot e\right|^{s - 2} \, d\sigma(e) \right) |\xi|^{s - 2}(1 + |\xi|)^{-\tau} \, d\xi. $$
Finally, one uses the fact that $s > 1$ to deduce the uniform bound
$$ \int_{\S^1} |\theta \cdot e|^{s - 2} \, d\sigma(e) \lesssim_{s} 1, \qquad  \theta \in \S^1, $$
which, recalling that $s < \tau$, proves that
$$ \int_{\R^{2}} |\pi_{2}(\xi)|^{s - 2}(1 + |\xi|)^{-\tau} \, d\xi \lesssim \int_{\R^{2}} (1 + |\xi|)^{(s - \tau) - 2} \, d\xi < \infty. $$
\end{proof}

\subsection{Proof of Theorem \ref{noDecay}}

Consider the sets $U_{i} \subset C[0,1]$, $i \in \N$, defined by
$$ U_{i} := \bigg\{f \in C[0,1] : \exists \, R = R_{f,i} \geq i \text{ such that } \inf_{\mu \in \cP(G_{f})} \sup_{i \leq |\xi| \leq R} |\widehat{\mu}(\xi)| > \tfrac{1}{5} \bigg\}. $$
Then:
\begin{itemize}
\item The sets $U_{i}$ are open in $C[0,1]$, see Proposition \ref{open} below.
\item The sets $U_{i}$ are dense in $C[0,1]$, see Proposition \ref{dense} below.
\end{itemize}
Given this information, we may infer that the set
$$\cR := \bigcap_{i \in \N} U_i \subseteq C[0,1]$$ 
is residual by definition. Now consider a function $f$ in $\cR$ and fix $\mu \in \cP(G_f)$.  For each $i \in \mathbb{N}$, we have that $\mu$ satisfies $|\widehat{\mu}(\xi)| \geq 1/5$ for some $\xi \in \R^{2}$ with $|\xi| \geq i$.  This means precisely that 
$$ \limsup_{|\xi| \to \infty} |\widehat{\mu}(\xi)| \geq \frac{1}{5},$$
which proves Theorem \ref{noDecay}.

To prove that the sets $U_{i}$ are open in $C[0,1]$, we need the following lemma:
\begin{lemma}\label{lum3} 
Let $g,h \in C[0,1]$. Consider the mapping $T_{g,h} \colon G_{g} \to G_{h}$ defined by
$$ T_{g,h}(x,g(x)) = (x,h(x)). $$
Then
$$  |\widehat{\mu}(\xi) - \widehat{T_{g,h\sharp}\mu}(\xi)| \leq 2\pi|\xi| \cdot \|g - h\|_{\infty} $$
for any $\mu \in \cP(G_{g})$.
\end{lemma}

\begin{proof}
Fix $\mu \in \cP(G_g)$ and write $\nu = T_{g,h\sharp} \mu$. Then
$$\mu = (\cdot,g(\cdot))_{\sharp}(\pi_{\sharp}\mu) \quad \text{ and} \quad \nu = (\cdot,h(\cdot))_\sharp (\pi_\sharp \mu),$$
where $\pi(x,y) = x$ is the projection onto the $x$-coordinate. Observing that the mapping $x \mapsto \exp(-2\pi i x \cdot \xi)$ is $2\pi|\xi|$-Lipschitz, we have the estimate
\begin{align*} |\widehat{\mu}(\xi) - \widehat{\nu}(\xi)| & = \Big|\int e^{-2\pi i x \cdot \xi} \, d\mu(x) - \int e^{-2\pi i x \cdot \xi} \, d\nu(x) \Big|\\
& \leq \int |\exp(-2\pi i(t,h(t)) \cdot \xi) - \exp(-2\pi i(t,g(t)) \cdot \xi) | \, d\pi_\sharp \mu (t)\\
& \leq  2\pi|\xi| \cdot \int |h(t)-g(t)|\, d\pi_\sharp \mu (t)  \\ 
&\leq 2\pi |\xi| \cdot \|h - g\|_{\infty}.
\end{align*}
\end{proof}

\begin{proposition}\label{open} 
The sets $U_{i}$ defined in the proof of Theorem \ref{noDecay} are open.
\end{proposition}

\begin{proof} Fix $i \in \N$ and $f \in U_{i}$, and let $R := R_{f,i} \geq i$ be the number appearing in the definition of $U_{i}$. Then, there exists $\epsilon > 0$ such that the following holds. For any measure $\nu \in \cP(G_{f})$ we may find a point $\xi \in \R^{2}$ such that $i \leq |\xi| \leq R$ and $|\widehat{\nu}(\xi)| > 1/5 + \epsilon$. With $\delta := \epsilon/(2\pi R) > 0$, we will now prove that $B(f,\delta) \subset U_{i}$. Fix $g \in B(f,\delta)$ and let $\mu \in \cP(G_{g})$ be arbitrary. We aim to find a point $\xi \in \R^{2}$ with $i \leq |\xi| \leq R$ such that $|\widehat{\mu}(\xi)| > 1/5$. First, note that $T_{f,g\sharp}\mu \in \mathcal{P}(G_{f})$, so there exists a point $\xi \in \R^{2}$ such that $i \leq |\xi| \leq R$ and
$$ |\widehat{T_{g,f\sharp}\mu}(\xi)| > \frac{1}{5} + \epsilon. $$
Lemma \ref{lum3} then shows that
$$ |\widehat{T_{g,f\sharp}\mu}(\xi) - \widehat{\mu}(\xi)| \leq 2\pi|\xi| \cdot \|g - f\|_{\infty} \leq 2\pi R \cdot \delta = \epsilon, $$
which gives
$$ |\widehat{\mu}(\xi)| \geq |\widehat{T_{g,f\sharp}\mu}(\xi)| - \epsilon > \frac{1}{5}. $$
This proves that $g \in U_{i}$ with $R_{g,i} := R$.
\end{proof}

\begin{proposition}\label{dense} 
The sets $U_{i}$ defined in the proof of Theorem \ref{noDecay} are dense in $C[0,1]$.
\end{proposition}

The proof is a combination of Lemmas \ref{denseG} and \ref{incG} below. We start by introducing dense classes $\cG_{M} \subset C[0,1]$, $M \in \N$, of `good functions'.

\begin{definition} 
Fix $M \in \N$. A function $g \in C[0,1]$ is a member of $\cG_{M}$, if the following holds for some $N \geq M$. The graph $G_{g}$ can be expressed as a union
$$ G_{g} = H_{g} \cup V_{g} $$
of a `horizontal piece' $H_{g}$ and a `vertical piece' $V_{g}$ such that $H_{g}$ and $V_{g}$ are compact, and
$$ \pi_{1}(V_{g}), \pi_{2}(H_{g}) \subset \{t \in \R : \cos(2\pi Nt) \geq 0.99 \text{ and } \cos(4\pi Nt) \geq 0.99\}. $$
Here $\pi_{i}$, $i \in \{1,2\}$, is the projection onto the $i^{\mathrm{th}}$ coordinate.
\end{definition}

\begin{lemma}
\label{denseG} 
The sets $\cG_{M}$ are dense in $C[0,1]$ for any $M \in \N$.
\end{lemma}

\begin{proof} 
Fix $M \in \N$, $f \in C[0,1]$ and $\epsilon > 0$. Pick some large integer $N \geq M$ and let $Q_{N} := \{k/N : k \in \Z\}$. For $1\leq k \leq N - 1$, define the intervals $I_{k} \subset [0,1)$ by
$$ I_{k} := \left[\frac{(k - 1)}{N}, \frac{k}{N}\right). $$
For $k = N$, let $I_{N} = [1 - 1/N,1]$. For $1 \leq k \leq N$, choose $q_{k} \in Q_{N}$ as close to $f(k/N)$ as possible, and define
$$ \tilde{g} := \sum_{k = 1}^{N} q_{k} \cdot \chi_{I_{k}}. $$
Since $f$ is uniformly continuous, we can ensure that $\|f - \tilde{g}\|_{\infty} < \epsilon$ and $|q_{k - 1} - g_{k}| < \epsilon$, $1 \leq k \leq N$, simply by choosing $N$ large enough. Also, we have
\begin{equation}\label{form5} \pi_{2}(G_{\tilde{g}}) \subset Q_{N} \subset \{t \in \R : \cos(2\pi Nt) = 1 = \cos(4\pi Nt)\}. \end{equation}
Now, the only problem is that $\tilde{g} \notin C[0,1]$. We correct the issue in the obvious way, by replacing the `jumps' between the consecutive values $q_{k}$ and $q_{k + 1}$ by nearly vertical affine patches, see Figure \ref{fig1}. The graph of the resulting function, $g$, then naturally divides into a `horizontal part' $H_{g} \subset G_{\tilde{g}}$, which is essentially the graph of $\tilde{g}$ (modulo shortening the intervals $I_{k}$ a bit), and a `vertical part' $V_{g}$, which is the union of the nearly vertical affine patches introduced above.

\begin{figure}[ht!]
\centering{\includegraphics[scale = 0.9]{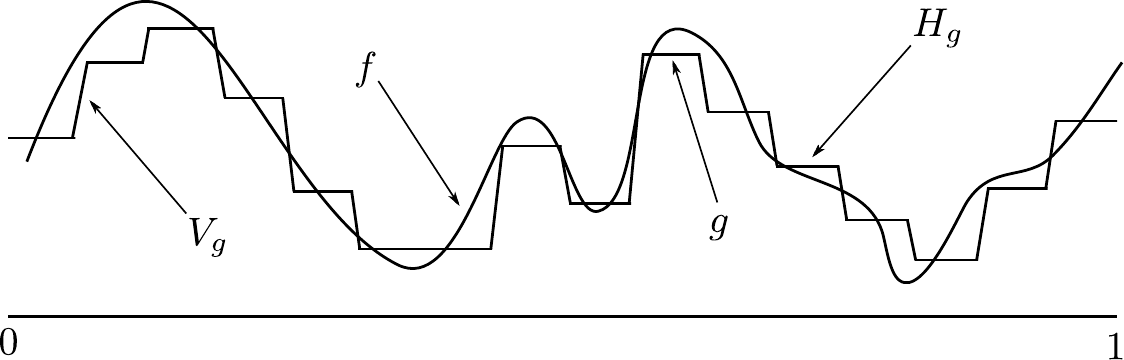}}
\caption{Constructing $g \in \cG_{M}$}\label{fig1}
\end{figure}

Moreover, given $\delta > 0$, the affine patches can be chosen so upright that $\pi_{1}(V_{g})$ is contained in the $\delta$-neighbourhood of $Q_{N}$ (if they could be taken completely vertical, we would have $\pi_{1}(V_{g}) \subset Q_{N}$). Choosing $\delta$ small enough then validates the inclusion
$$ \pi_{1}(V_{g}) \subset \{t \in \R : \cos(2\pi Nt) \geq 0.99 \text{ and } \cos(4\pi Nt) \geq 0.99\}. $$
The analogous inclusion for $\pi_{2}(H_{g})$ is implied by \eqref{form5}. Finally, it follows from the inequalities $\|f - \tilde{g}\|_{\infty} < \epsilon$ and $|q_{k - 1} - q_{k}| < \epsilon$, $1 \leq k \leq N$, that $\|f - g\|_{\infty} < 2\epsilon$. This completes the proof of the lemma.
\end{proof}

\begin{lemma}
\label{incG} 
The inclusion $\cG_{M} \subset U_{M}$ holds.
\end{lemma}

\begin{proof} 
Fix $g \in \cG_{M}$, let $N \geq M$ be the number given in the definition of $\cG_{M}$, and pick a measure $\mu \in \cP(G_{g})$. Since $\mu(G_{g}) = 1$, we must have either $\mu(H_{g}) \geq 1/2$ or $\mu(V_{g}) \geq 1/2$. Suppose that 
\begin{equation}\label{form4} \mu(H_{g}) \geq \frac{1}{2}. \end{equation}
We decompose the measure $\mu$ into horizontal and vertical pieces as follows: $\mu = \mu_{H} + \mu_{V}$, where $\mu_{H} = \mu|_{H_{g}}$ and $\mu_{V} = \mu - \mu_{H}$. We infer from \eqref{form4} that $\mu_{H}(\R^{2}) \geq 1/2$ and $\mu_{V}(\R^{2}) \leq 1/2$. Now, the key observation is that the $\pi_{2}$-projection of the support of $\mu_{H}$ lies in the set 
$$\{t \in \R : \cos(2\pi Nt) \geq 0.99 \text{ and } \cos(4\pi Nt) \geq 0.99\}$$ 
(if we had \eqref{form4} for $V_{g}$, we would use the $\pi_{1}$-projection instead). This yields the estimate
\begin{align*} |\widehat{\mu}(0,N) + \widehat{\mu}(0,2N)| & \geq \Rea (\widehat{\pi_{2\sharp}\mu}(N) + \widehat{\pi_{2\sharp}\mu}(2N))\\
& = \Rea (\widehat{\pi_{2\sharp}\mu_{H}}(N) + \widehat{\pi_{2\sharp}\mu_{H}}(2N))\\
&\qquad + \Rea (\widehat{\pi_{2\sharp}\mu_{V}}(N) + \widehat{\pi_{2\sharp}\mu_{V}}(2N))\\
& = \int_{\R} \cos(2\pi Nt) + \cos(4\pi Nt) \, d\pi_{2\sharp}\mu_{H}(t)\\
&\qquad + \int_{\R} \cos(2\pi Nt) + \cos(4\pi Nt) \, d\pi_{2\sharp}\mu_{V}(t) \geq 0.99 - \frac{9}{16} > \frac{2}{5}, \end{align*} 
which means that either $|\widehat{\mu}(0,N)| > 1/5$ or $|\widehat{\mu}(0,2N)| > 1/5$. On the last line, we used the fact that 
$$\cos(2\pi Nt) + \cos(4\pi Nt) \geq -\frac{9}{8},$$
which can be seen by first applying the trigonometric identity $\cos 2\alpha = 2\cos^{2} \alpha - 1$ to the latter summand and then making the substitution $\cos(2\pi Nt) \mapsto u$ and minimizing the resulting polynomial.

We have now deduced that no matter how we choose the measure $\mu \in \cP(G_{g})$, we always have 
$$\max\{|\widehat{\mu}(0,N)|,|\widehat{\mu}(N,0)|,|\widehat{\mu}(0,2N)|,|\widehat{\mu}(2N,0)|\} \geq  c > 1/5$$ 
with $c = 0.99 - 9/16$. This means that $g \in U_{M}$ with the choice $R_{g,M} = 2N$.
 \end{proof}
 
 This concludes the proofs of Proposition \ref{dense} and of Theorem \ref{noDecay}.

%\begin{proof}[Proof of Proposition \ref{dense}] Combine the previous two lemmas.
%\end{proof}

\section{Extensions and open questions}

\subsection{Extensions to functions defined on closed subsets of $[0,1]$} \label{spaceE}
In this section, $E \subset [0,1]$ is an arbitrary closed set, and $G_{f} = \{(x,f(x)) : x \in E\} \subset \R^{2}$ for functions $f \colon E \to \R$.  

(1) The proof of Theorem \ref{attainable} shows, without modifications, that $\Fd G_{f} \leq 1$ for functions $f \colon E \to \R$. Moreover, if $\mu \in \cP(G_{f})$ is any measure, the projection $\mu_{1} \in \cP(E)$ of $\mu$ onto the $x$-coordinate satisfies
$$\widehat{\mu_{1}}(t) = \widehat{\mu}(t,0), \qquad t \in \R.$$
These observations combined imply that $\Fd G_f \leq \Fd E.$

(2) As mentioned in the proof of Proposition \ref{attainable}, Kaufman's work \cite{kaufman} implies that for any $0 \leq s \leq 1$ there exists a closed set $E \subset [0,1]$ and a function $f \in C(E)$ such that the graph $G_f$ has
$$\Fd G_f = s.$$
This still leaves open the following question:
\begin{question}
For any closed set $E \subset [0,1]$ and for any $0 \leq s \leq \Fd E$, does there exist a function $f \in C(E)$ with
$$\Fd G_f = s?$$
\end{question}

(3) Theorem \ref{noDecay} can be generalised for typical functions in $f \in C(E)$. Consider the sets $U_i \subset C[0,1]$ in the proof of Theorem \ref{noDecay}. Define the projection $\Phi : C[0,1] \to C(E)$ by
$$\Phi(f)(x) = f(x), \quad x \in E, \quad f \in C[0,1]$$
Then $\Phi$ is a surjective continuous linear operator (surjectivity follows from the `affine extension' argument in the proof of Proposition \ref{attainable}), and so by the open mapping theorem \cite[Theorem 2.11]{rudin} $\Phi$ is an open map.  It follows that $\Phi(U_i) \subset C(E)$ is open and dense for each $i \in \mathbb{N}$ and so
$$\cR := \bigcap_{i \in \N} \Phi(U_i) \subseteq C(E)$$ 
is residual by definition. Now consider a function $f$ in $\cR$ and fix $\mu \in \cP(G_f)$.  For each $i \in \mathbb{N}$, $f=\Phi(g_i)$ for some $g_i \in U_i$ and so, for each $i$, $\mu \in \cP(G_{g_i})$. Since $g_i \in U_i$, we have that $\mu$ satisfies $|\widehat{\mu}(\xi)| \geq 1/5$ for some $\xi \in \R^{2}$ with $|\xi| \geq i$.  This means precisely that 
$$ \limsup_{|\xi| \to \infty} |\widehat{\mu}(\xi)| \geq \frac{1}{5},$$
which shows that we can replace $[0,1]$ with $E$ in Theorem \ref{noDecay}.

\subsection{Prevalent Fourier dimension of graphs}

A further direction one could take this work in is to consider an alternative notion of genericity.  Here we use typicality, but it is equally natural to use the theory of \emph{prevalence}.  This notion is more measure theoretic than topological in nature and as such often gives very different answers to genericity questions.  For example, the \emph{typical} Hausdorff dimension of a graph of a function in $C[0,1]$ is one, but the \emph{prevalent} Hausdorff dimension is two.  These two answers are as different as possible and give a good indication of the stark differences in the two theories.  Prevalence was formulated by Hunt, Sauer and Yorke in the mid 90s \cite{prevalence1} and has been used to study the generic dimensions of graphs of continuous functions on numerous occasions, see \cite{BH, horizon, fraserhyde, lowerprevalent, mcclure, images}.  The most general result to date is essentially due to Bayart and Heurteaux \cite{BH}, although their result was slightly extended in \cite{images}, and states that for an arbitrary uncountable closed set $E \subset [0,1]$, the Hausdorff dimension of the graph of a prevalent function $f \in C(E)$ is as large as possible, namely $\Hd E + 1$.  We ask the following question:

\begin{question}
What is the prevalent Fourier dimension of $G_f$ for $f \in C(E)$?
\end{question}

Recalling the discussion in Section \ref{spaceE}, the answer can be at most $\Fd E \leq \Hd E$, which, combined with the with the fact that prevalently $\Hd G_{f} = \Hd E + 1$ mentioned above, shows that the graph of the prevalent function in $C(E)$ is not a Salem set. So, what should we expect from the Fourier dimension of the graph of a prevalent function? Since for the other dimensions, the prevalent dimension of a graph is as big as possible, one might conjecture the same to be true for Fourier dimension.  A possible proof strategy would be to follow the work of Bayart and Heurteaux and use fractional Brownian motion (or some other suitable stochastic process) on $E$ to obtain lower bounds for the prevalent Fourier dimension.

\end{document}